\documentclass[11pt]{amsart}

\usepackage{amsfonts,epsfig}
\usepackage{latexsym}
\usepackage{amssymb}
\usepackage{amsmath}
\usepackage{amsthm}
\usepackage{graphics}
\usepackage{verbatim}
\usepackage{array,booktabs,ctable}
\usepackage[backref]{hyperref}
\usepackage{color}
\definecolor{mylinkcolor}{rgb}{0.8,0,0}
\definecolor{myurlcolor}{rgb}{0,0,0.8}
\definecolor{mycitecolor}{rgb}{0,0,0.8}
\hypersetup{colorlinks=true,urlcolor=myurlcolor,citecolor=mycitecolor,linkcolor=mylinkcolor,linktoc=page,breaklinks=true}
\usepackage{bigstrut}

\newtheorem{defn}{Definition}[section]

\newtheorem{lemma}[defn]{Lemma}

\newtheorem{theorem}[defn]{Theorem}

\newtheorem{proposition}[defn]{Proposition}

\theoremstyle{definition}

\newtheorem{remark}[defn]{Remark}

\newtheorem*{acknowledgments}{Acknowledgments}

\newcommand{\Q}{\mathbb Q}

\newcommand{\Z}{\mathbb Z}

\newcommand{\F}{\mathbb F}

\newcommand{\PP}{\mathbb P}

\newcommand{\rk}{\operatorname{rk}}

\newcommand{\Frob}{\operatorname{Frob}}

\newcommand{\Gal}{\operatorname{Gal}}

\makeatletter
\newcommand\suchthat{%
 \@ifstar
  {\mathrel{}\middle|\mathrel{}}
  {\mid}%
}
\makeatother

\def\diam#1{\langle#1\rangle}

\begin{document}
\bibliographystyle{plain}
\title{Tamagawa Numbers of elliptic curves with $C_{13}$ torsion over quadratic fields}

\author{Filip Najman}
\address{Department of Mathematics\\ University of Zagreb\\ Bijeni\v cka cesta 30\\ 10000 Zagreb\\ Croatia}
\email{fnajman@math.hr}

\thanks{The author gratefully acknowledges support from the QuantiXLie Center of Excellence.}

\begin{abstract}
Let $E$ be an elliptic curve over a number field $K$, $c_v$ the Tamagawa number of $E$ at $v$, and let $c_E=\prod_{v}c_v$. Lorenzini proved that $v_{13}(c_E)$ is positive for all elliptic curves over quadratic fields with a point of order $13$. Krumm conjectured, based on extensive computation, that the $13$-adic valuation of $c_E$ is even for all such elliptic curves. In this note we prove this conjecture and furthermore prove that there is a unique such curve satisfying $v_{13}(c_E)=2$.
\end{abstract}

\maketitle

\section{Introduction}

Let $K$ be a number field and $E$ an elliptic curve defined over $K$. For every finite prime $v$ of $K$, denote by $K_v$ the completion of $K$ at $v$ and by $k_v$ the residue field of $v$. The subgroup $E_0(K_v)$ of $E(K_v)$ consisting of points that reduce to nonsingular points in $E(k_v)$ has finite index $E(K_v)$ and one defines the \textit{Tamagawa number} of $E$ at $v$ to be this index $c_v:=[E(K_v):E_0(K_v)]$.

We define $c_E$ to be
$$c_{E/K}:=\prod_{v}c_v.$$
It will always be clear from the context which number field $K$ we are working over, so for brevity's sake, we will write $c_E$ instead of $c_{E/K}$.

Because the ratio $c_E/\#E(K)_{tors}$ appears as a factor in the leading term of the $L$-function of $E$, by the Birch-Swinnerton--Dyer conjecture, it is natural to study how the value of $c_E$ depends on $E(K)_{tors}$. Many results describing how the value $c_E$ depends on $E(K)_{tors}$ have been obtained by Lorenzini \cite{lor} for elliptic curves over $\Q$ and quadratic fields and by Krumm in his PhD thesis \cite[Chapter 5]{kru} for number fields of degree up to $4$.

Let us give a short explanation of how $c_E$ can depend on $E(K)_{tors}$. Suppose for simplicity that $N=\#E(K)_{tors}$ is prime. Let $E_1(K_v)$ be the subgroup of $E(K_v)$ of points which reduce to the point at infinity in $E(k_v)$ and let $E_{ns}(k_v)$ be the group of nonsingular points in $E(k_v)$.  There exists an exact sequence of abelian groups
$$0\longrightarrow E_1(K_v) \longrightarrow E_0(K_v) \longrightarrow E_{ns}(k_v) \longrightarrow 0.$$
If $v$ does not divide $N$, then there are no points of order $N$ in $E_1(K_v)$, as $E_1(K_v)$ is isomorphic to the formal group of $E$. If $v$ is also small enough such that there cannot be any points of order $N$ in $E_{ns}(k_v)$, due to the Hasse bound, then it follows that $E_0(K_v)$ does not have a point of order $N$. It then follows, by definition, that $N$ has to divide $c_v$.

Throughout the paper $C_n$ will denote a cyclic group of order $n$. Using the argument above, Krumm showed (taking $v$ to be a prime above $2$ and $N=13$) that for all elliptic curves $E$ over all quadratic fields $K$ with $E(K)_{tors}\simeq C_{13}$, the value $c_E$ is divisible by $169$ \cite[Proposition 5.3.4.]{kru}. Furthermore, he conjectured that for all elliptic curves with torsion $C_{13}$ over quadratic fields the value $v_{13}(c_E)$ is even \cite[Conjecture 5.3.6.]{kru}. The conjecture is true for the 48925 such elliptic curves that he tested.

In this note we prove this conjecture. More explicitly, we prove the following theorem.
\begin{theorem}
\label{main_tm}
Let $E$ be an elliptic curve over a quadratic field $K$ with $E(K)_{tors}\simeq C_{13}$. Then $v_{13}(c_E)$ is a positive even integer.
\end{theorem}

Finally, in Theorem \ref{zavrsni_tm} we show that there is a unique elliptic curve $E$ over any quadratic field $K$ such that $v_{13}(c_E)=2$.

\section{Elliptic curves with $C_{13}$ torsion over quadratic fields}

As we are looking at elliptic curves with $C_{13}$ torsion, we are naturally led to studying the modular curve $X:=X_1(13)$. The $K$-rational points on the modular curve $Y_1(13)$ correspond to isomorphism classes of pairs $(E,P)$, where $E/K$ is an elliptic curve and $P\in E(K)$ is a point of order $13$. The compactification $X_1(13)$ of $Y_1(13)$ has genus~2 and in particular is hyperelliptic. There are six rational cusps on $X_1(13)$, representing N\'eron $13$-gons, and six cusps with field of definition $\Q(\zeta_{13})^+=\Q(\zeta_{13}+\zeta_{13}^{-1})$, the maximal real subfield of $\Q(\zeta_{13})$, representing N\'eron 1-gons. For more on the moduli interpretation of the cusps of $X_1(n)$ see \cite[Chapter II]{dr} or \cite[Section 9]{di}. The diamond automorphism
$\iota:=\diam{5}=\diam{-5}$ which acts as $i((E,\pm P))=(E, \pm 5P)$ is the hyperelliptic involution; we note that the fixed points of~$\iota$ lie outside the cusps.

The $\Q$-rational points on $X$ are all cusps. Both $X$ and $J$ have bad reduction only at 13. For a prime $v$ of $\Q(\zeta_{13})^+$ not dividing 13, the cusps in $X(\Q(\zeta_{13})^+)$ reduce bijectively to the cusps of $\widetilde{X}(\overline{k_v})$, where $\widetilde{X}$ is the reduction of $X$ modulo $v$.

Reduction mod $v$ is injective on $J(\Q(\zeta_{13})^+)_{tors}$ for all primes $v$ of $\Q(\zeta_{13})^+$ not dividing 13 - this follows from \cite[Appendix]{kat} for $v\nmid 2$. For $v\mid 2$, injectivity follows from the fact that $J(\Q(\zeta_{13})^+)$ has no 2-torsion (this can easily be checked in Magma \cite{magma}).

The Jacobian $J:=J_1(13)$ has rank 0 over $\Q$ and $J(\Q)\simeq \Z / 19 \Z$ - this fact was originally proved by Mazur and Tate in \cite{mt}. We will need the rank of $J$ over $\Q(\zeta_{13})^+$.

\begin{lemma} \label{rnk_J13}
The rank of $J(\Q(\zeta_{13})^+)$ is $0$.
\end{lemma}
\begin{proof}
We use $2$-descent, as implemented in the \texttt{RankBound()} function in Magma \cite{magma}, to prove this Lemma. As Magma is unable, in reasonable time, to perform a 2-descent on $J$ directly over $\Q(\zeta_{13})^+$, we do the following.

Let $F$ be the cubic subfield of $\Q(\zeta_{13})^+$; then $[\Q(\zeta_{13})^+:F]=2$ and $\Q(\zeta_{13})^+=F(\sqrt{13})$. Thus we have that $\rk J(\Q(\zeta_{13})^+) = \rk J(F) + \rk J^{13}(F)$, where $J^{13}$ denotes the Jacobian of the quadratic twist of $X_1(13)$ by $13$. Magma computes $\rk J(F)=\rk J^{13}(F)=0$, proving the claim.
\end{proof}

\begin{remark}
See \cite{kn} for a different proof, using the fact that $X_1(13)$ is bielliptic, of the fact that rank of $J(\Q(\zeta_{13})^+)$ is $0$.
\end{remark}

From \cite{bbdn,rab} it follows that all elliptic curves with $C_{13}$ torsion over quadratic fields are of the form
\begin{equation}\label{jed_ek}
 E_t: y^2+axy+cy=x^3+bx^2,
\end{equation}
where
\begin{equation}\label{jed_ek2}
\begin{aligned}
a&=\frac{(t-1)^2(t^2+t-1)s-t^7+2t^6+3t^5-2t^4-5t^3+9t^2-5t+1}{2},\\
b&=\frac{t(t-1)^2((t^5+2t^4-5t^2+4t-1)s-t^8-t^7+4t^6+2t^5+t^4-13t^3+14t^2-6t+1)}{2},\\
c&=t^5b,
\end{aligned}
\end{equation}
for some $t \in \Q$, and where
\begin{equation}\label{jed_s}
s=\sqrt{t^6-2t^5 +t^4-2t^3+6t^2-4t+1}.
\end{equation}
The curve $E_t$ is defined over $\Q(s)$.

We will need the following two lemmas.
\begin{lemma}
\label{lem_hyp}
If $(E,P)=x\in X(K)$ is a non-cuspidal point, where $K$ is a quadratic field. Then $\iota(x)=x^\sigma$, where $\sigma$ is the generator of $\Gal(K/\Q)$. In particular $E$ is isomorphic over $K$ to $E^\sigma$.
\end{lemma}
\begin{proof}
See \cite[Chapter 4]{bbdn} or \cite[Theorem 2.6.9.]{kru}.
\end{proof}

\begin{lemma} \label{lem_mult_red}Let $E$ be an elliptic curve over a quadratic field $K$ with $E(K)_{tors}\simeq C_{13}$. Let $v$ be a prime such that $13$ divides $c_{v}$. Then $E$ has split multiplicative reduction at $v$.
\end{lemma}
\begin{proof}
This is well known, see for example \cite[Corollary C.15.2.1. p.447.]{sil}.
\end{proof}

\begin{lemma} Let $(E,P)=x\in X(K)$, let $v$ be a prime of $K$ such that $v\nmid 13$ and $13|c_{v}$, let $p$ be the rational prime below $v$ and let $v'$ be a prime of $\Q(\zeta_{13})^+$ above $p$. Then $x$ mod $v$ is equal to $C$ mod $v'$ for a cusp $C\in X(\Q(\zeta_{13})^+)$ such that $C$ mod $v'$ is $\F_p$-rational.
\label{rat_cusps}
\end{lemma}
\begin{proof}
Denote by $\widetilde{x}$ the reduction od $x$ mod $v$ and denote by $\overline{C}$ the reduction of $C\in X(\Q(\zeta_{13})^+)$ mod $v'$. First note that $\widetilde{x}=\overline{C}$, for some cusp $C\in X(\Q(\zeta_{13})^+)$, follows from the fact that all the cusps of $X$ are defined over $\Q(\zeta_{13})^+$ and that the cusps in $X(\Q(\zeta_{13})^+)$ reduce bijectively mod $v'$ to the cusps of $X(\overline {\F_p})$.

We now divide the proof in two cases: when $p\equiv \pm 1 \pmod{13}$ and when $p \not\equiv \pm 1 \pmod {13}$.

\noindent\textsc{Case 1: $p \not\equiv \pm 1 \pmod {13}$}\\
We claim that $\widetilde{x}=\overline C$ represents a N\'eron $13$-gon, from which it follows that $C$ represents a N\'eron $13$-gon; thus $C$ is defined over $\Q$ and $\overline{C}$ is $\F_p$-rational.

Suppose the opposite - that $\widetilde{x}=\overline{C}$ represents a N\'eron $1$-gon. Hence $P$ specializes to the identity component of the special fiber of the N\'eron model of $E$ at $v$. Since $\widetilde{P}$ is of order $13$ and $E$ has, by Lemma \ref{lem_mult_red}, split multiplicative reduction at $v$, it follows that $13$ divides the order of the multiplicative group $\mathbb G_m$ over $k_v$, from which it follows that that $13$ divides $p^2-1$, which is a contradiction with our assumption.

\noindent\textsc{Case 2: $p \equiv \pm 1 \pmod {13}$}\\
Since $p \equiv \pm 1 \pmod {13}$, it follows that $p$ splits completely in $\Q(\zeta_{13})^+$, the field over which all the cusps of $X$ are defined. Hence it follows that for every cusp $C$ of $X$, $\overline{C}$ is $\F_p$-rational.
\end{proof}
\begin{remark}
Note that the case $K\subseteq \Q(\zeta_{13})^+$ (i.e. $K=\Q(\sqrt{13})$) is not possible, since there exist no elliptic curves with $C_{13}$ torsion over $\Q(\sqrt{13})$ by \cite[Theorem 3]{kamnaj}.
\end{remark}

\section{Proof of Theorem \ref{main_tm}}

\begin{proposition}
Let $E_t$ be an elliptic curve over a quadratic field $K$ with $E_t(K)_{tors}\simeq C_{13}$. Let $v$ be a prime of $K$ over a rational prime $p$ such that $13$ divides $c_{v}$. Then $p$ splits in $K$.
\label{prop1}
\end{proposition}

\begin{proof} We split the proof into two cases, when $v$ divides $13$ and when it does not.

\noindent\textsc{Case 1: $v$ does not divide $13$}\\
Let $x$ be a non-cuspidal point on $X(K)$ and let $v'$ be a prime of $\Q(\zeta_{13})^+$. Denote by $\widetilde{y}$ the reduction of a $y\in X(K)$ mod $v$ and denote by $\overline y$ the reduction of a $y\in X(\Q(\zeta_{13})^+)$ mod $v'$.

Note that $X(\Q)$ consists purely of cusps, so $x$ is not defined over $\Q$. Suppose $p$ is inert or ramified, i.e. $p=v$ or $p=v^2$. Let $\widetilde{x}=\overline{C}$ and $\widetilde{x^\sigma}=\overline{C_\sigma}$, for some cusps $C$ and $C_\sigma$; $\overline{C}$ and $\overline{C_\sigma}$ are $\F_p$-rational by Lemma \ref{rat_cusps}.

Recall that $k_v\simeq \F_p$ if $p$ is split or ramified and $k_v\simeq \F_{p^2}$ if $p$ is inert and that the generator $\Frob v$ of $\Gal(k_v/\F_p)$ is non-trivial if and only if $p$ is inert. We have that, for a general $x\in K$, $\widetilde{x^\sigma}=\widetilde{x}^{\Frob v}\neq \widetilde{x}$ if $p$ is inert, $\widetilde{x^\sigma}\neq\widetilde{x}^{\Frob v}= \widetilde{x}$ if $p$ splits and $\widetilde{x^\sigma}=\widetilde{x}^{\Frob v}= \widetilde{x}$ if $p$ is ramified.
If $p$ is inert or ramified, it follows that
\begin{equation}\label{frob}
\overline{C_\sigma}=\widetilde{x^\sigma}=\widetilde{x}^{\Frob v}=\overline{C}^{\Frob v}=\overline{C}.
\end{equation}
It follows that $[\widetilde{x}+\widetilde{x^\sigma}-2\overline{C}]=0$. Since $[x+x^\sigma-2C]$ is a $\Q(\zeta_{13})^+$-rational divisor class, and hence a torsion point by Lemma \ref{rnk_J13}, injectivity of reduction mod $v'$ on $J(\Q(\zeta_{13})^+)_{tors}$ implies that $[x+x^\sigma-2C]=0$. Thus $x+x^\sigma-2C$ is a divisor of a rational function $g$, and since $x,x^\sigma \neq C$, $g$ is of degree $2$. Since the hyperelliptic map is unique (up to an automorphism of $\PP^1$), it follows that $g:X\rightarrow X/\diam{\iota}\simeq \PP^1$ is the same as quotienting out by the hyperelliptic involution $\iota$. Thus $\iota$ permutes the zeros and permutes the poles of $g$, from which it follows that $C$ is fixed by $\iota$, which we know is not true.

\noindent\textsc{Case 2: $v$ divides $13$}\\
As every elliptic curve $E_t$ with $C_{13}$ torsion over a quadratic fields is of the form given in \eqref{jed_ek} and \eqref{jed_ek2}, it is clear that the reduction type of $E$ over a prime $v$ over 13 depends only on the value of $t$ mod $13$, if $v_{13}(t)\geq0$. An easy computation shows that $E_t$ has multiplicative reduction only if $v_{13}(t)\geq 0$ and $t\equiv 0,1 \pmod{13}$ or if $v_{13}(t)<0$. In all these cases, $13$ splits in $\Q(s)$.
\end{proof}

\begin{proposition}
\label{prop_gl}
Let $E$ be an elliptic curve over a quadratic field $K$ with $E(K)_{tors}\simeq C_{13}$. Let $v$ be a prime over $p$ such that $13$ divides $c_{v}$ and $\sigma$ the generator of $\Gal(K/\Q)$. Then $v\neq v^\sigma$ and $c_v(E)=c_{v^\sigma}(E)$.
\end{proposition}
\begin{proof}
By Proposition \ref{prop1}, $v\neq v^\sigma$. By Lemma \ref{lem_hyp} it follows that $E^\sigma \simeq E$, and hence $c_v(E)=c_{v^\sigma}(E^\sigma )=c_{v^\sigma}(E).$
\end{proof}
From Proposition \ref{prop_gl}, it is clear that $v_{13}(\prod_{v}c_v)$ is even. The fact that $v_{13}(c(E))>0$ follows from \cite[Proposition 1.3.]{lor}, proving Theorem \ref{main_tm}.

\section{The elliptic curve with smallest $c_E$}

Since we have $169|c_{E_t}$, it is natural to ask how many curves $E_t$ with $C_{13}$ torsion over quadratic fields satisfy $v_{13}(c_{E_t})=2$. In \cite[Example 5.3.5]{kru} Krumm found a single curve satisfying $v_{13}(c_{E_t})=2$. In fact, this curve satisfies $c_{E_t}=169$. We prove that this curve is the unique curve having this property.

\begin{theorem}
\label{zavrsni_tm}
The elliptic curve
\begin{equation}\label{jed_ek3}
E_2:y^2 +xy + y = x^3 - x^2 + \frac{-541 +131\sqrt{17}}{2}x + 3624-879\sqrt{17}
\end{equation}
is the only elliptic curve $E$ over any quadratic field with $C_{13}$ torsion such that $v_{13}(c_E)=2$; for all other such curves $13^4|c_E$.
\end{theorem}
\begin{proof}
We will show that for all curves not isomorphic to $E_2$ with torsion $C_{13}$ over quadratic fields, $13^4|c_E$. Let $E_t$ be an elliptic curve with $C_{13}$ torsion over a quadratic field, then $E_t$ is of the form given by \eqref{jed_ek} and \eqref{jed_ek2}, for some $t \in \Q$ and where $s$ is given in \eqref{jed_s}. Then
$$\Delta(E_t)=\frac{t^{13}(t-1)^{13}(t^3-4t^2+t+1)f(t,s)}{2},$$
where $f(t,s)$ is a degree $42$ polynomial, and
\small{$$j(E_t)=\frac{(t^2 - t + 1)^3(t^{12} - 9t^{11} + 29t^{10} - 40t^9 + 22t^8 - 16t^7 + 40t^6 - 22t^5 - 23t^4 + 25t^3 - 4t^2 - 3t + 1)^3}{t^{13}(t-1)^{13}(t^3-4t^2+t+1)}.$$}

By \cite[Corollary 3.4]{lor}, if $\wp$ is a prime over 2, then $13|c_\wp$ and it follows from Proposition \ref{prop1} (or \cite[Theorem 2.6.9.]{kru}) that $2$ splits in $K$.

Let $\wp$ be a prime of $K$ not dividing 2. Suppose that $m:=v_\wp(t)>0$. Then it follows that $E/K$ has split multiplicative reduction modulo $\wp$ of type $I_{13m}$ (see \cite[Section 2.2]{lor} for details how to check this explicitly). This implies that $v_{13}(c_\wp(E))>0$ by \cite[Theorem VII.6.1]{sil}. By Proposition \ref{prop_gl}, it follows that $\wp\neq \wp^\sigma$ and that $E/K$ also has split multiplicative reduction modulo $\wp^\sigma$ and that $c_\wp(E)= c_{\wp^\sigma}(E)$.

Now suppose that $m:=-v_\wp(t)>0$; then \eqref{jed_ek} is not an integral model of $E_t$ at $\wp$, but the equation with invariants {\small
$$\begin{aligned}
a_1 &=\frac{(s - 1)z^7 + (-3s + 5)z^6 + (2s - 9)z^5 + (s + 5)z^4 + (-s + 2)z^3 -
        3z^2 - 2z + 1}{2},\\
a_2 &=\frac{z^8(z-1)^2((s - 1)z^8 + (-4s + 6)z^7 + (5s - 14)z^6 + 13z^5 + (-2s - 1)z^4 +
        (-s - 2)z^3 - 4z^2 + z + 1)}{2},\\
a_3 &=\frac{z^3(z-1)^2((s - 1)z^8 + (-4s + 6)z^7 + (5s - 14)z^6 + 13z^5 + (-2s - 1)z^4 +
        (-s - 2)z^3 - 4z^2 + z + 1)}{2},
\end{aligned}$$}
with $z=t^{-1}$ is integral in $\wp$. We compute that
$$\Delta(E_t)=\frac{z^{13}(z-1)^{13}(z^3+z^2-4z+1)g(z,s)}{2},$$
where $g(z,s)$ is a polynomial of degree 42. Again we obtain, using the same arguments as before, that $E_t$ has split multiplicative reduction of type $I_{13m}$ at both $\wp$ and $\wp^\sigma$.

The same argument as before shows that if $v_\wp(t-1)\neq 0$, then $c_\wp(E)= c_{\wp^\sigma}(E)=13m$ for some positive integer $m$, and that $\wp \neq \wp^\sigma$.

Thus if we want $v_{13}(c_E)=2$, the primes above 2 are the only primes $\wp $ such that $13|c_\wp$, which implies that the primes $\wp$ above 2 are the only primes $\wp$ such that $v_\wp(t)\neq 0$ or $v_\wp(t-1)\neq 0$. We see that the only possibilities are $t=-1,\frac 1 2$ and $2$. All three values give the same curve $E_2$.
\end{proof}

\begin{remark}
We expect there to be infinitely many nonisomorphic elliptic curves $E_t$ with $C_{13}$ torsion over quadratic fields such that $v_{13}(c_{E_t})=4.$

To see this, notice that $v_{13}(c_{E_t})=4$ if
\begin{itemize}
  \item[1)] If we put $t=r/s$, then $rs(r-s)$ has exactly $2$ prime divisors,
  \item[2)] For all primes $\wp$ of $\mathcal O_K$,  $v_{\wp}\left((t^3-4t^2+t+1)f(t,s)\right)\neq13 k,$ for $k\in \Z -\{0\}.$
\end{itemize}

Condition 1) is true for $\{|r|,|s|,|r-s|\}=\{1,2^p-1, 2^p\}$, for $p\neq 13$ such that $2^p-1$ is prime, i.e a Mersenne prime, for $\{|r|,|s|, |r-s|\}=\{1, 2^k, 2^k+1\}$, where $2^k+1$ is a Fermat prime, or for
$\{|r|,|s|, |r-s|\}=\{1,8,9\}$. Conjecturally, there exist infinitely many Mersenne primes (and finitely many Fermat primes).

 Heuristically, we expect condition $2)$ to be satisfied very often as there is no reason to expect the appearance of $13k$-th powers in the prime factorization of the numerator or denominator of $(t^3-4t^2+t+1)f(t,s)$. Together with the (conjectural) infinitude of Mersenne primes, this should heuristically imply that there exists infinitely many values $t$ such that $v_{13}(c_{E_t})=4$.
\end{remark}

\begin{acknowledgments}
We are grateful to Dino Lorenzini for bringing Krumm's conjecture to our attention and for many helpful comments and suggestions. We thank Peter Bruin and Matija Kazalicki for helpful conversations and the anonymous referee for many valuable comments and suggestions that greatly improved the paper.
\end{acknowledgments}

\end{document}